\documentclass[11pt]{amsart}
\usepackage{amsmath}
\usepackage{amsfonts}
\usepackage{amssymb}
\usepackage{mathtools}
\newtheorem{thm}{Theorem}[section]

\usepackage[capitalize]{cleveref}
\newtheorem{cor}[thm]{Corollary}

\newtheorem{lemma}[thm]{Lemma}
\newtheorem{prop}[thm]{Proposition}

\theoremstyle{definition}

\newtheorem{remark}[thm]{Remark}

\newtheorem{conj}[thm]{Conjecture}

\newcommand\blfootnote[1]{%
  \begingroup
  \renewcommand\thefootnote{}\footnote{#1}%
  \addtocounter{footnote}{-1}%
  \endgroup
}

\newcommand{\bb}[1]{\mathbb{#1}}

\newcommand{\I}{\operatorname{I}}
\newcommand{\tr}{\operatorname{tr}}
\newcommand{\id}{\operatorname{id}}

\newcommand{\ran}{\operatorname{ran}}

\begin{document}

\title[]{Compositions of PPT Maps}
\author[M. Kennedy]{Matthew Kennedy}
\address{Department of Pure Mathematics, University of Waterloo, Waterloo, ON, Canada N2L 3G1}
\email{matt.kennedy@uwaterloo.ca}
\author[N. A. Manor]{Nicholas A. Manor}
\address{Department of Pure Mathematics, University of Waterloo, Waterloo, ON, Canada N2L 3G1}
\email{nmanor@uwaterloo.ca}
\author[V.~I.~Paulsen]{Vern I.~Paulsen}
\address{Institute for Quantum Computing and Department of Pure Mathematics, University of Waterloo,
Waterloo, ON, Canada  N2L 3G1}
\email{vpaulsen@uwaterloo.ca}


\begin{abstract} M. Christandl conjectured that the composition of any trace preserving PPT map with itself is entanglement breaking.
We prove that Christandl's conjecture holds asymptotically by showing that the distance between the iterates of any unital or trace preserving PPT map and the set of entanglement breaking maps tends to zero. Finally, for every graph we define a one-parameter family of maps on matrices and determine the least value of the parameter such that the map is variously, positive, completely positive, PPT and entanglement breaking in terms of properties of the graph. Our estimates are sharp enough to conclude that Christandl's conjecture holds for these families.
\end{abstract}

%
%

\blfootnote{First author supported by NSERC Grant Number 418585}
\blfootnote{Second author supported by NSERC Grant Number 396164132}

\maketitle

\section{General Introduction} The usual mathematical model for a {\it quantum channel} is a {\it completely positive trace-preserving} (\textit{CPTP}) map between two matrix spaces. A  completely positive map is called a {\it positive partial transpose} (\textit{PPT}) map if the composition of the map with the transpose map on the range space is still completely positive.  PPT maps play an important role in the study of entanglement. 

A completely positive map can also be identified with a state on the tensor product of the two matrix algebras, and the states corresponding to PPT maps are called PPT states. PPT states could play a role in {\it quantum key distribution(QKD)}, which is the study of the use of various quantum mechanical systems to construct shared states that would be used to insure secure communication.


These considerations lead Christandl to consider questions about how PPT maps behaved under composition and lead to the following conjecture \cite{RJKHW}: 

\begin{conj}[PPT-Squared Conjecture]
The composition of a pair of PPT maps is always entanglement breaking.
\end{conj}

We will define and discuss entanglement breaking maps in the following section. From the point of view of shared states, this conjecture is equivalent to the following statement \cite{BCHW}:  

``Assume that Alice and Charlie share a PPT state and that Bob and Charlie share a PPT state; then the state of Alice and Bob, conditioned on any measurement by Charlie, is always separable." 

The PPT-squared conjecture is known to be true for maps on the $2 \times 2$ matrices.  We prove that Christandl's intuition about the behaviour of PPT maps under composition is at least asymptotically true by showing that the distance between the iterates of a PPT channel and the set of entanglement breaking maps tends to zero.  From the point of view of a quantum communication network this implies that if each pair of channels shares the same PPT state then, eventually, the network will behave as if they are sharing a state corresponding to an entanglement breaking map.

In Section 2, we give precise definitions of the concepts introduced. In Section 3, we show using basic techniques from the theory of topological semigroups that the conjecture holds asymptotically.  In Section 4, we consider a new type of spectral graph theory problem; we associate a one-parameter family of maps to each graph and determine in terms of the graph the smallest values of the parameter for which the map is, variously, positive, completely positive, PPT, and entanglement breaking.  Our estimates are sharp enough that we are able to show that whenever two maps in this family are PPT, then their composition is entanglement breaking, i.e., the conjecture holds for this family.

\section{Basics}
We recall that by Choi’s theorem \cite[Theorem~2]{C}, a map $\phi : M_p \to M_q$ is CP if and only if its Choi matrix $C_{\phi} := \big( \phi(E_{i,j}) \big)$ is a positive semidefinite matrix in $M_p(M_q) = M_{pq}$.
 
 We present here a very basic but useful result about PPT maps.
 
\begin{lemma}
	Let $\phi : M_p \to M_q$ be a CP map. Then $\phi$ is PPT if and only if $\phi \circ T$ is CP, where $T$ denotes the transpose map on $M_p$. In other words, to check whether a CP map is PPT we may compose or precompose with the transpose map.
\end{lemma}

\begin{proof}
	The Choi matrix of $T \circ \phi$ is
	\begin{align*}
		C_{T \circ \phi} = 
		\begin{pmatrix}
		\phi(E_{11})^T &\cdots &\phi(E_{1n})^T \\
		\vdots		    &\ddots &\vdots \\
		\phi(E_{n1})^T &\cdots &\phi(E_{nn})^T		
		\end{pmatrix}
		 =
		\begin{pmatrix}
		\phi(E_{11}) &\cdots &\phi(E_{n1}) \\
		\vdots		    &\ddots &\vdots \\
		\phi(E_{1n}) &\cdots &\phi(E_{nn})
		\end{pmatrix}^T = {C_{\phi \circ T}}^T.
	\end{align*}
	Since the transpose is a positive map, $C_{\phi \circ T} $ is positive if and only if its transpose $C_{T \circ \phi}$ is.
\end{proof}

\begin{cor}\label{composition}
The set of PPT maps is closed under composition by CP maps on the right and on the left.
\end{cor}

 By results of \cite{HRS}, $\phi$ is \textit{entanglement breaking} if and only if it can be written as
 \[ \phi(X) = \sum_k v_kw_k^*Xw_kv_k^* = \sum_k s_k(X) P_k,\]
 for some set of vectors  $ w_k  \in \bb C^p$ and $v_k \in \bb C^q$, where $s_k(X) = \langle w_k, Xw_k \rangle$ and $P_k = v_kv_k^*$. By normalizing the $w_k$'s and $v_k$'s we extract weights $d_k = \|w_k|| \cdot \|v_k\|$, and we may assume the $s_k$ are states and the $P_k$'s are rank one projections.
 
 Thus, $\phi$ is entanglement breaking if and only if $C_{\phi}$ can be written as
 \[ C_{\phi} = \sum_k d_k \big( s_k(E_{i,j}) P_k \big) = \sum d_k Q_k \otimes P_k,\]
 with the $Q_k \in M_p$ density matrices and the $P_k \in M_q$ rank one projections.
 
 Moreover since every density matrix can be written as a sum of rank one projections, we have that $\phi$ is entanglement breaking if and only if
 \[ C_{\phi} = \sum_l t_l R_l \otimes S_l,\]
 where $R_\ell \in M_p$ and $S_\ell \in M_q$ are rank one projections, and $t_\ell$ are positive weights.

Another more recent characterization, given in \cite{JKPP}, is that $\phi$ is entanglement breaking exactly when it factors through $\ell^\infty_k$ via positive maps, for some $k$. More precisely: there are positive maps $\psi : M_p  \to \ell^\infty_k$ and $\gamma : \ell^\infty_k \to M_q$ so that $\phi = \gamma \circ \psi$.

\begin{remark}\label{abelian}
This characterization may be modified slightly so that instead of factoring through a finite-dimensional abelian C$^*$-algebra, any abelian C$^*$-algebra may be used. Simply note that, if $\phi = \gamma \circ \psi$ for $\psi: M_p \to C(X)$ and $\gamma : C(X) \to M_q$, then $\phi \circ \Theta$ is still completely positive for any positive map $\Theta$ on $M_p$. This simply follows from the fact that positive maps into or from an abelian C$^*$-algebra are necessarily completely positive.
\end{remark}
 
 
\section{The Asymptotic Result}

It turns out that Christandl's intuition holds asymptotically in the following sense: the sequence of iterates of a PPT channel $\phi$ on $M_n$ approaches the set of entanglement breaking maps. To prove this we use some very basic results from the theory of abelian semigroups. Since these objects live in a finite dimensional space, convergence is independent of any particular metric.

We first examine the case of an idempotent unital PPT map.

\begin{lemma}
Let $\phi : M_n \to M_n$ be an idempotent unital PPT map. Then the range of $\phi$ is an abelian C*-algebra with respect to the product $a \ast b = \phi(ab)$, for $a,b$ in $\phi(M_n)$.
\end{lemma}

\begin{proof}
A result of Choi and Effros \cite{CE} implies that the range of $\phi$ is a C*-algebra with respect to the product given above. We must show that this C*-algebra is abelian.

Supposing it is non-abelian: it has a direct summand isomorphic to $M_k$ for some $k\geq2$. So by composing with the associated projection we get, by \cref{composition}, an induced PPT map $\psi$ which is surjective onto $M_k$. In this case, $\psi \otimes \id_k : M_n \otimes M_k \to M_k \otimes M_k$ is also surjective. So there is a positive matrix $(A_{ij}) \in M_n \otimes M_k$ with $\psi \otimes \id_k ((A_{ij})) = (E_{ij})$, however $(T(E_{ij}))$ is not positive, so the composition $T \circ \psi$ is not completely positive, contradicting $\psi$ being PPT. Therefore, we conclude that the range of $\phi$ is abelian.

\end{proof}

\begin{prop}\label{idempotent}
Let $\phi$ be an idempotent unital PPT map on $M_n$. Then $\phi$ is entanglement breaking.
\end{prop}

\begin{proof}
By the previous lemma, $\phi$ factors through a finite dimensional \\
abelian C*-algebra; so by \cref{abelian} we have the result.
\end{proof}

\begin{lemma}
If $\phi$ is a contractive map on $M_n$ endowed with any norm, then there is an idempotent map $\psi$ in the limit points of $(\phi^k)_{k\geq 1}$.
\end{lemma}

\begin{proof}
Let $S$ denote the closure of $\{ \phi^k : k \in \bb{N} \}$. This is compact since $\phi$ is contractive, and it is an abelian semigroup under composition. Let $K = \cap_{a \in S} \ aS$ be the intersection of all singly generated ideals. Then we claim $K$ is a minimal ideal in $S$.

First of all, this set is non-empty by the finite intersection property and by the fact that the product of finitely many ideals is contained in their intersection. It is also clearly an ideal, as an intersection of ideals. Minimality follows from the fact that every ideal contains a singly generated ideal, and of course every singly generated ideal contains $K$.

We will now show that $K$ has a multiplicative identity, thus giving us an idempotent in $S$. Take $k \in K$. Since $K$ is minimal we have $k^2 S = K$, so there is $s \in K$ such that $(sk) k = s k^2 = k$. We claim that $sk$ is the identity in $K$.

Taking any $k' \in K$, again by minimality there is $s' \in S$ such that $s' k = k'$. From this we get that $(sk)k' = (sk)s'k = s'(sk)k = s'k = k'$. So we may take $\psi  = sk$.
\end{proof}

Note that if $\phi$ is trace preserving then so is the idempotent $\psi$. The next lemma tells us exactly how $\phi^k$ approaches the set of EB maps.

\begin{lemma}
If $\phi$ and $\psi$ are as above then
\begin{align*}
\| \phi^k - \phi^k \circ \psi \| \to 0.
\end{align*}
\end{lemma}

\begin{proof}
Since $\psi$ is a limit of powers of $\phi$, $\phi$ and $\psi$ commute. Hence $\ran(\psi)$ and $\ker(\psi) = \ran(\id - \psi)$ are invariant for $\phi$. It follows from the spectral mapping theorem that $\sigma(\phi |_{\ran(\psi)}) \subseteq \bb{T}$ and $\sigma(\phi|_{\ran(\id - \psi)}) \subseteq \bb{D}$. Since $\phi$ and $\psi$ commute, this implies $\sigma(\phi^n|_{\ran(\id - \psi)}) \subseteq \bb{D}$. Hence $\lim \|\phi^n - \phi^n \circ \psi\| = 0$.
\end{proof}

\begin{thm}
Every unital or trace preserving PPT map $\phi$ is asymptotically entanglement breaking, in the sense that $d(\phi^k, EB) \to 0$, where EB denotes the set of entanglement breaking maps on $M_n$.
\end{thm}

\begin{proof}
In the unital case, we know that the idempotent map $\psi$ from the lemma above is PPT and hence entanglement breaking by \cref{idempotent}. So for every $n$ the map $\phi^k \circ \psi$ is entanglement breaking, and by the previous lemma this implies $d(\phi^k,EB) \to 0$.

To retrieve the trace preserving case, recall that trace preserving maps are precisely the adjoints of unital ones, and the set of entanglement breaking maps is also *-symmetric. So, since the adjoint operation is an isometry we get the result.
\end{proof}

\begin{remark}  The above theorem can also be deduced from the work of Lami and Giovanetti\cite{LG} on {\it asymptotically entanglement-saving channels}.  A channel is called asymptotically entanglement-saving if no limit point of its iterates is entanglement breaking, which is easily seen to be equivalent to the negation of our condition that $\lim_k d( \phi^k, EB) =0$. Combining \cite[Theorem~32.2]{LG}  and \cite[Theorem~12]{LG} shows that no PPT map can be asymptotically entanglement-saving. 
\end{remark}

 \section{Schur Product Maps}
 In this section, we examine a class of maps for which it is possible to prove Christandl's conjecture; these form a one parameter family of maps defined from graphs. Determining for which values of the parameter these maps belong to the various classes of positive maps, completely positive maps, PPT maps, and entanglement breaking maps, leads to interesting spectral questions in combinatorics.  For some of these families of maps we are able to answer these questions exactly, for others we can only give estimates on the parameter. However, our estimates are good enough to show that, for all graphs, the PPT-squared conjecture is true for maps in this family.
 
  Recall that given matrices $A=(a_{i,j}), B= (b_{i,j})$ of the same size, their Schur product is the matrix
  \[ A \circ B := (a_{i,j}b_{i,j}).\]
  Given $A \in M_p$ then we set $S_A:M_p \to M_p$ to be the map $S_A(B) = A \circ B$.
  
  It is well known that $S_A$ is CP if and only if $A \ge 0$.
   
 \begin{prop} Let $P$ be an $n \times n$ matrix. Then $S_P$ is PPT if and only if $P\ge 0$ and $P$ is diagonal.
 \end{prop}
 \begin{proof} It is readily checked that if $P \ge 0$ and $P$ is diagonal then $S_p$ is PPT.
 
 Assume that $S_P$ is PPT. Then since $S_P$ is CP, $P \ge 0$. Let $T$ denote the transpose map, and assume that $T \circ S_P$ is CP. If $P=(p_{i,j})$, then by Choi's theorem,
 \[ \sum_{i,j} E_{i,j} \otimes p_{i,j} E_{j,i}  \ge 0.\]
 If $i \ne j$, then the $2 \times 2$ block submatrix
 \[ \begin{pmatrix} p_{i,i} E_{i,i} & p_{i,j} E_{j,i} \\ \overline{p_{i,j}} E_{i,j} & p_{j,j} E_{j,j} \end{pmatrix} \ge 0.\]
 But this is possible, only if $p_{i,j} =0$. Hence $P$ is diagonal.
 \end{proof}
   
   Thus, there are no ``interesting" Schur product maps that are PPT.
 
 Let $Tr:M_p \to \bb C$ be the usual trace, and let $tr(B) = \frac{1}{p} Tr(B)$ denote the normalized trace.
 
 We set $\delta:M_p \to M_p$ to be $\delta(X) = tr(X) I_p$.  Note that $\delta \circ \delta = \delta$.
 Note that this CP map is entanglement breaking.
 
 Now let $A= A^*$ be a $p \times p$ matrix of 0's and 1's with the diagonal equal to 0.  The set of $(i,j)$ such that $a_{i,j} \ne 0$ can be thought of as the edge set of a graph $G = (V,E)$ on p vertices, in which case $A$ is the adjacency matrix of the graph.  The Schur product map $S_A$ is idempotent.
 
 We are interested in the one parameter family of maps,
 \[ \gamma_t = \gamma_{t,A} = t \delta +S_A,\] and in determining the following parameters of the graph $G$:
 \begin{itemize}
 \item $t_{pos} = \min \{ t: \gamma_t \text{ is a positive map } \},$
 \item $t_{cp} = \min \{ t: \gamma_t \text{ is CP } \}$,
 \item $t_{ppt} = \min \{ t: \gamma_t \text{ is PPT } \}$,
 \item $t_{eb}= \min \{ t: \gamma_t \text{ is EB } \}.$
 \end{itemize}
 Clearly, we have that
 \[ t_{pos} \le t_{cp} \le t_{ppt} \le t_{eb}.\]
 
 In general, we expect $t_{pos} < t_{cp}$.  In fact for the the case of the complete graph on 2 vertices, i.e., an edge we have that
 \[ \gamma_t \big( \begin{pmatrix} p_{11} & p_{12}\\ p_{21} & p_{22} \end{pmatrix} \big) = \begin{pmatrix} \frac{t(p_{11}+p_{22})}{2} & p_{12} \\p_{21} & \frac{t(p_{11}+p_{22})}{2} \end{pmatrix}.\]
 Using the determinant test and the root mean inequality, one sees that this map is positive for $t=1,$ since for any positive matrix,
 \[ \frac{p_{11}+p_{22}}{2}^2 \ge p_{11}p_{22} \ge p_{12}p_{21}.\]
 However, one readily sees that for $t=1$, the Choi matrix of this map is not positive. Thus, $\gamma_1$ is positive but not CP.
 
 Note that 
 \[ \gamma_t \circ \gamma_t = \gamma_{t^2},\]
 so that the PPT-squared conjecture will hold for this family of maps if and only if $t_{eb} \le t_{ppt}^2,$ which we shall prove below.
 
 
 

 Given an adjacency matrix $A$ we let $\lambda_{min} $ to denote the least (real) eigenvalue of $A$. Since $Tr(A)=0$ this number will always be strictly negative, as long as $A \ne 0$.
 
 \begin{prop} If $A$ is a non-zero $p \times p$ adjacency matrix, then $t_{cp}=t_{ppt}= -p \lambda_{min}$.
 \end{prop}
 \begin{proof}
To compute $t_{cp},$ we must determine restrictions imposed by requiring the Choi matrix of the map is positive. This matrix is  $\frac{t}{p} \operatorname{I} \otimes \operatorname{I} + C_{S_A}$ where $C_{S_A}$ is the Choi matrix of the map $S_A$.  For $t_{ppt}$ we also need that $\frac{t}{p} \operatorname{I} \otimes \operatorname{I} + C_{S_A\circ T}$ is positive, where $T$ is the transpose map on $M_p$ and $C_{S_A \circ T}$ is the Choi matrix of $S_A \circ T$. 

As for requiring $\frac{t}{p} \operatorname{I} \otimes \operatorname{I} + C_{S_A}$ to be positive, we find the minimal eigenvalue of $C_{S_A} = \sum_{(i,j) \in E(G)} E_{ij} \otimes E_{ij}$, where $G$ is the associated graph of $A$. Notice that $C_{S_A}$ is identically zero on the space spanned by $e_k \otimes e_l$, with $k \neq l$, and on the span of $e_k \otimes e_k$ it behaves exactly as $A$ acting on $\bb{C}^p$. We have that $-\lambda_{min} \operatorname{I} \otimes \operatorname{I} + C_{S_A \circ T}$ is positive, and it is non-positive for any strictly smaller multiple of $\operatorname{I} \otimes \operatorname{I}$. Thus, $t_{cp}= -p \lambda_{min}$.

For the second case, observe that $C_{S_A \circ T} = \sum_{(i,j) \in E} E_{ji} \otimes E_{ij}$, so that $(C_{S_A \circ T})^2 = \sum_{(i,j) \in E} E_{ii} \otimes E_{jj}$ is a diagonal matrix of only 1's and 0's. In particular, the spectrum of $(C_{S_A \circ T})^2$ must be a subset of $\{0,1\}$, but then $C_{S_A \circ T}$ may only have eigenvalues of -1, 0 and 1. So for $t=p$ we will certainly have that $\gamma_{t,A} \circ T$ is completely positive, and it is minimal exactly when -1 is an eigenvalue of $C_{S_A \circ T}$. In fact, this will always be the case; choose $(k,l)$ such that $a_{k,l} =1$ and notice that $C_{S_A \circ T} (e_k \otimes e_l - e_l \otimes e_k) = e_l \otimes e_k - e_k \otimes e_l$.

Thus,  $t_{ppt} = p \cdot min \{ 1, - \lambda_{min} \}$. It is easily checked that for any non-zero adjacency matrix, $\lambda_{min} \le -1$ so that $t_{ppt} = -p \lambda_{min}$, also.

To verify this last claim, note that if $A_{i,j}=1$ and we set $v= \frac{e_i - e_j}{\sqrt{2}}$ then $v$ is a unit vector and
with $\langle Av,v \rangle = -1$, from which it follows that $\lambda_{min} \le -1$.
\end{proof}


We would now like to fully understand  $t_{eb}$, although this will pose a greater issue as it is rarely clear when a matrix is separable in a tensor product. We present here a natural upper bound for $t_{eb}$ via a simple computation. Recall that we view $A$ as being the adjacency matrix of a graph $G=(V,E)$.

\begin{lemma}
Let $A$ be the adjacency matrix of a graph $G = (V,E)$ on $p$ vertices. Then  $\gamma_{pd} = pd \cdot \delta + S_A$ is entanglement breaking, where $d$ denotes the maximum edge degree in $G$. In particular, $t_{eb} \le pd$.
\end{lemma}
 
\begin{proof} 
 We proceed by considering the Choi matrix $C_\phi$ of $\phi = pd \cdot \delta + S_A$ and showing it is separable. It is easy to see that 
\begin{align*} C_\phi & = d \I_p \otimes \I_p + \sum_{(i,j) \in E} E_{ij} \otimes E_{ij} \\
& = D + \sum_{(i,j) \in E \ \text{and} \ i<j} E_{ii} \otimes E_{jj} + E_{jj} \otimes E_{ii} + E_{ij} \otimes E_{ij} + E_{ji} \otimes E_{ji}, 
\end{align*}
where $D$ is a diagonal matrix consisting of 1's and 0's (hence it is separable). So it suffices to show that matrices of the same form as the summand on the right hand side (where $i$ and $j$ may vary from 1 to $p$) are separable.
  
 We use only the four following positive matrices in $M_p$ to prove this fact:
 \begin{eqnarray*}
 && Q_{1,i,j} = E_{i,i} + E_{j,j} + E_{ij} + E_{ji} \\
 && Q_{2,i,j} = E_{i,i} + E_{j,j} + E_{ij} - E_{ji} \\
 && Q_{3,i,j} = E_{i,i} +E_{j,j} + iE_{ij} - iE_{ji} \\
 && Q_{4,i,j} = E_{i,i} +E_{j,j} - iE_{ij} + iE_{ji},
 \end{eqnarray*}
 where $i<j$ vary from 1 to $p$. A routine computation shows that

 \begin{multline*}
 4\big( (E_{i,i}+E_{j,j}) \otimes (E_{i,i}+E_{j,j})+ E_{ij} \otimes E_{ij} + E_{ji} \otimes E_{ji} \big)
 = \\ Q_{1,i,j} \otimes Q_{1,i,j} + Q_{2,i,j} \otimes Q_{2,i,j} 
 + Q_{3,i,j} \otimes Q_{4,i,j} + Q_{4,i,j} \otimes Q_{3,i,j}.
 \end{multline*}
 Summing over all edges we get that
 \[R= 4 \sum_{(i,j) \in E} (E_{i,i}+E_{j,j}) \otimes (E_{i,i} +E_{j,j}) + 8 \sum_{(i,j) \in E} E_{i,j} \otimes E_{j,i}, \]
 is separable.
 Now in the sum $\sum_{(i,j) \in E} (E_{i,i}+E_{j,j}) \otimes (E_{i,i} +E_{j,j}) $, for $k \ne l$ each term $E_{k,k} \otimes E_{l,l}$ appears at most once, while $E_{i,i} \otimes E_{i,i}$ occurs exactly $2 d_i\le 2 d$ times.
 Since each term $E_{k,k} \otimes E_{l,l}$ is separable, we see that we can add a separable term $Q$ to $R$ so that
 \[R+Q = 8d \big( \sum_{i,j} E_{i,i} \otimes E_{j,j} \big) + 8 \sum_{(i,j) \in E} E_{i,j} \otimes E_{j,i} = 8 C_{\phi},\]
 and it follows that $C_{\phi}$ is separable so that $\phi= \gamma_{pd,A}$ is entanglement breaking. 
  
\end{proof}



 
Before stating the next result note that the Schur product $A \circ B$ of two adjacency matrices is again an adjacency matrix.

\begin{cor}
	If $A$ and $B$ are adjacency matrices and the maps $\gamma_{t_1,A}$ and $\gamma_{t_2,B}$ are PPT, then their composition is entanglement breaking.
\end{cor}
\begin{proof}
	The composition evaluates to $\gamma_{t_1 t_2,A \circ B}$, so if either $A$ or $B$ is zero then the composition is the map
	\begin{align*}
		X \mapsto t_1 t_2 \tr(X).
	\end{align*}	 
	This map is clearly entanglement breaking.
	
	If both are non-zero matrices then $t_1 t_2 \geq p^2$, and this is necessarily greater than $t_{eb}$ for any adjacency matrix of size $p$ since the degree of any vertex cannot exceed $p-1$. The corollary follows.
\end{proof}

Numerically, it is possible to compute $t_{pos}$. Indeed, to check if $\gamma_t$ is positive it is enough to check that it is positive for all rank one positive matrices arising from unit vectors. For such a matrix we have that
$\gamma_t((\alpha_i \overline{\alpha_j})) = \frac{t}{p} \I + S_A((\alpha_i \overline{\alpha_j}))$.  This leads to
\[ t_{pos} = -p \min \{ \lambda_{min}(S_A(\alpha_i \overline{\alpha_j})) : |\alpha_1|^2+ \cdots + |\alpha_p|^2 =1 \}.\]

On the other hand, if we consider $S_A$ as a map from $M_p$ to $M_p$ endowed with its trace norm, i.e., the Schatten one-norm, then the norm $\|S_A\|_1$ is attained on such rank one matrices, and so $t_{pos} \le p \|S_A\|_1$. However, the adjoint of $S_A$ is again the map $S_A$, so that one has $\|S_A\|_1= \|S_A\|,$ where the latter norm is the norm of the linear map $S_A: M_p \to M_p$ and $M_p$ is endowed with its usual operator norm, i.e.,  $\|X\|^2 = \lambda_{max}(X^*X)$. There is a well-known formula for computing the norm of such Schur product maps. See for example \cite[Theorem 8.7]{Pa}. Thus,  $t_{pos} \le p \|S_A\|$ gives an upper bound on this quantity.

Next we turn to some lower bounds on $t_{pos}$ in terms of more familiar graph parameters.

\begin{prop}  Let $A$ be the adjacency matrix of a graph $G=(V,E)$ on $p$ vertices and let $\vartheta= \vartheta(G)$ denote the Lov\'asz theta number of the graph and let $\overline{\vartheta}= \vartheta(\overline{G})$, denote the Lov\'asz theta number of the graph complement of $G$. Then
\[ t_{pos} \ge \max \{ 1, -\lambda_{min}(A), \frac{-p \lambda_{min}(A)}{|E|}, \frac{-p \lambda_{min}(A)}{t_{eb}}, 
\frac{\lambda_{max}(A)}{ \overline{\vartheta} -1}   \}.\]
\end{prop}
\begin{proof}  Let $r \ge t_{pos},$ so that $\gamma_r$ is a positive map.

The first two inequalities come from applying $\gamma_r$ to the positive matrices $-\lambda_{min}(A) \I +A$ and the $p \times p$ matrix of all 1's.    The third inequality comes from applying $\gamma_r$ to the Laplacian matrix of the graph, $L = A + \sum_i d_i E_{i,i}$, which is positive since it is diagonally dominant.

To see the fourth inequality, note that if $s \ge t_{eb}$, then $\gamma_r \circ \gamma_s = \gamma_{rs}$ is the composition of a positive map and an entanglement breaking map and so is CP. Hence,  $t_{eb} t_{pos} \ge t_{cp}= - p \lambda_{min}(A)$, and the inequality follows.

For the final inequality, we use the fact that \cite[Theorem 3]{Lo},
\[\overline{\vartheta} = \min \{ \lambda_{max}(H): H=H^*, S_{\I+A}(H) = I\I +H \}.\] 
Let $K=K^*$ be the matrix such that $S_{\I +A}(K) = \I+A$ and $\lambda_{max}(K) = \overline{\vartheta}$.
Then we have that $K \le \overline{\vartheta} \I$ and so,  $r \I +A = \gamma_r(K) \le r \overline{\vartheta} \I$. Hence,
$A \le r(\overline{\vartheta} -1) \I$ and the last inequality follows.
\end{proof}

\end{document}